\documentclass[10pt]{article}
\font\smallit=cmti10
\font\smalltt=cmtt10

\usepackage{amssymb,latexsym,amsmath,epsfig,amsthm} 

\makeatletter

\renewcommand\section{\@startsection {section}{1}{\z@}
{-30pt \@plus -1ex \@minus -.2ex}
{2.3ex \@plus.2ex}
{\normalfont\normalsize\bfseries\boldmath}}

\renewcommand\subsection{\@startsection{subsection}{2}{\z@}
{-3.25ex\@plus -1ex \@minus -.2ex}
{1.5ex \@plus .2ex}
{\normalfont\normalsize\bfseries\boldmath}}

\renewcommand{\@seccntformat}[1]{\csname the#1\endcsname. }

\makeatother

\newtheorem{theorem}{Theorem}

\newtheorem{proposition}{Proposition}

\theoremstyle{definition}
\newtheorem{definition}{Definition}
\newtheorem{conjecture}{Conjecture}

\newtheorem{example}{Example}

\usepackage{url}

\begin{document}

\begin{center}
\uppercase{\bf \boldmath A binary version of the Mahler--Popken complexity function}
\vskip 20pt
{\bf John M. Campbell}\\
{\smallit Department of Mathematics and Statistics, Dalhousie University, Halifax, Nova Scotia, Canada}\\
{\tt jmaxwellcampbell@gmail.com}\\ 
\end{center}
\vskip 20pt
\centerline{\smallit Received: 05/04/24, Revised: 08/08/24, Accepted: 09/18/24, Published: --/--/--} 
\vskip 30pt

\centerline{\bf Abstract}
\noindent
 The (Mahler--Popken) \emph{complexity} $\| n \|$ of a natural number $n$ is the smallest number of ones that can be used via combinations of multiplication 
 and addition to express $n$, with parentheses arranged in such a way so as to form legal nestings. We generalize $\| \cdot \|$ by defining $\| n \|_{m}$ 
 as the smallest number of possibly repeated selections from $\{ 1, 2, \ldots, m \}$ (counting repetitions), for fixed $m \in \mathbb{N}$, that can be used to 
 express $n$ with the same operational and bracket symbols as before. There is a close relationship, as we explore, between $\|\cdot\|_{2}$ and lengths of 
 shortest addition chains for a given natural number. This illustrates how remarkable it is that $(\| n \|_{2} : n \in \mathbb{N} )$ is not currently included in the 
 On-Line Encyclopedia of Integer Sequences and has, apparently, not been studied previously. This, in turn, motivates our exploration of the complexity 
 function $\| \cdot\|_{2}$, in which we prove explicit upper and lower bounds for $\|\cdot\|_{2}$ and describe some problems and further areas of research 
 concerning $\|\cdot\|_{2}$. 

\pagestyle{myheadings}
\markright{\smalltt INTEGERS: 24 (2024)\hfill}
\thispagestyle{empty}
\baselineskip=12.875pt
\vskip 30pt

\section{Introduction}\label{sectionIntroduction}
 The notion of \emph{integer 
 complexity} was given implicitly by Mahler and Popken in 1953 \cite{MahlerPopken1953}. Subsequent to Guy's work concerning the complexity of integers 
 \cite[\S F26]{Guy1994} \cite{Guy1986}, notable advances in the study of the integer complexity function are due Altman et al.\ 
 \cite{AltmanPhD,Altman2018,Altman2015,Altman2016,Altman2019,AltmanZelinsky2012}. These past references have inspired our explorations based on 
 generalizing the concept of the complexity of an integer in a naturally number-theoretic or combinatorial way. In this paper, we introduce a family of 
 generalizations of the Mahler--Popken complexity function that may be thought of as providing an analogue of integer partitions of bounded width. In 
 particular, the ``binary'' case of our generalizations, as described below, provides a natural next step forward from the definition of integer complexity, and 
 is also closely related to what is meant by the length of a shortest \emph{addition chain}. 

 For a natural number $n$, the integer complexity function $ \|\cdot\|$ may be defined so that $ \| n \|$ is equal to the least number of ones required to 
 express $n$ with any combination of multiplication, addition, and bracket symbols, provided that any brackets are grouped in legal nestings. 

\begin{example}
 We find that $\| 6 \| = 5$, writing $6 = (1+1) \times (1+1+1)$. 
\end{example}

 Fundamental properties concerning the function $ \|\cdot\|$ include the inequalities 
\begin{equation}\label{originalinequalities}
 \| n + m \| \leq \| n \| + \| m \| 
 \ \ \ \text{and} \ \ \ \| n \times m \| \leq \| n \| + \| m \|, 
\end{equation}
 along with the explicit bounds whereby 
\begin{equation}\label{boundsoriginal}
 3 \log_{3}(n) \leq \| n \| \leq 3 \log_{2}(n) 
\end{equation}
 for $n > 1$. If we consider how $ \|\cdot\|$ may be thought of as encoding how ``complicated'' a natural number is in terms of the given constraints for 
 expressing natural numbers with the specified symbols, this lends itself toward how related notions of complexity could arise in the study of integer partitions 
 with entries other than $1$ allowed. In this regard, Mahler and Popken \cite{MahlerPopken1953} considered the inverse mapping associated with the 
 size of the largest number expressible using a fixed number of copies of a real number $x$ and with the same operations as for $\| \cdot \|$. 
 This leads us to consider new complexity functions that allow us to express a given natural number using more than one number, 
 in contrast to $\| \cdot \|$ and to the work of Mahler and Popken \cite{MahlerPopken1953}. 

\begin{definition}\label{definitionmary}
 We define the \emph{$m$-ary complexity function} $\|\cdot\|_{m}$ so that $\| n \|_{m}$, for a natural number $n$, equals the smallest number of possibly 
 repeated selections from $\{ 1, 2, \ldots, m \}$ (counting repetitions) that can be used to express $n$ with the same operational and bracket symbols 
 involved in our definition of $ \|\cdot\|$, and again subject to legal nestings. 
\end{definition}

\begin{example}
 The $m = 1$ case of Definition \ref{definitionmary}
 is such that $ \|\cdot\|_{1} = \|\cdot\|$. 
\end{example}

 In addition to the work of Guy \cite[\S F26]{Guy1994} \cite{Guy1986} and by Altman et al.\ \cite{AltmanPhD,Altman2018,Altman2015,
Altman2016,Altman2019,AltmanZelinsky2012}, for further research that concerns the integer complexity function and that motivates our generalization 
 $ \|\cdot\|_{m}$ of $\|\cdot\|$, see \cite{AriasdeReyna2000,AriasdeReyna2024,CernenoksIraidsOpmanisOpmanisPodnieks2015,
CordwellEpsteinHemmadyMillerPalssonSharmaSteinerbergerVu2019,Rawsthorne1989,Steinerberger2014}. 

\section{Addition Chains and Binary Complexity}\label{sectionAddition}
 It is well known that addition chains, as defined below, may be seen as providing something of a secondary version of complexity \cite[p.\ 2]{AltmanPhD}. 
 Since the function $\|\cdot\|_{2}$ may similarly be seen as a secondary version of integer complexity, this leads us to consider how the integer sequence $ 
 (\| n \|_{2} : n \in \mathbb{N} )$ relates to known properties associated with addition chains. 

\begin{definition}
 An \emph{addition chain} for $n \in \mathbb{N}$ 
 is a tuple $(a_{0}, a_{1}, \ldots, a_{r})$ satisfying $a_{0} = 1$ and $a_{r} = n$ and such that: For all $k \in \{ 1, 2, \ldots, 
 r \}$, there exist indices $i, j \in \{ 0, 1, \ldots, k-1 \}$ whereby $a_{k} = a_{i} + a_{j}$. 
\end{definition}

\begin{example}\label{shortest20}
 We may verify that a shortest addition chain ending with $20$ is 
\begin{equation}\label{displaychain}
 (1, 2, 4, 5, 10, 20). 
\end{equation}
 Informally, we can think of this shortest addition chain as being in correspondence with the binary complexity given by the decomposition 
\begin{equation}\label{suggestcorrespond2}
 (2 \times 2 + 1) \times 2 \times 2. 
\end{equation}
 Informally, the natural numbers we obtain from successive subsequences of natural numbers and binary operations, read from left to right in 
 \eqref{suggestcorrespond2}, are $2$, $4$, $5$, $10$, and $20$, and this agrees with the sequence integers in the tuple in \eqref{displaychain}. 
\end{example}

\begin{definition}\label{definitionlength}
 The \emph{length} of an addition chain is equal to $-1$ plus the number of its entries. 
 We let $\ell(n)$ denote the length of a shortest addition chain ending with $n$. 
\end{definition}

\begin{example}
 The length of the addition chain in \eqref{displaychain} is $5$, and this agrees with the value of $\| 20 \|_{2}$, with regard to the correspondence 
 suggested between \eqref{displaychain} and \eqref{suggestcorrespond2}. 
\end{example}

 The entry in the On-Line Encyclopedia of Integer Sequences (OEIS) \cite{oeis} providing the length of a shortest addition chain is indexed as {\tt A003313} and is of 
 particular interest for our purposes, due to how closely this sequence relates to $(\| n \|_{2} : n \in \mathbb{N} )$. In this regard, the first point of 
 disagreement between $(\| n \|_{2} : n \in \mathbb{N}_{\geq 2} )$ and 
 $( \text{{\tt A003313}}(n) : n \in \mathbb{N}_{\geq 2} ) = (\ell(n) : n \in \mathbb{N}_{\geq 2})$ 
 is illustrated below and occurs at the $n = 23$ point: 
\begin{align*}
 & (\| n \|_{2} : n \in \mathbb{N}_{\geq 2} ) = \\
 & (1, 2, 2, 3, 3, 4, 3, 4, 4, 5, 4, 5, 5, 5, 4, 5, 5, 6, 5, 6, 6, 7, \ldots), \\
 & ( \ell(n) : n \in \mathbb{N}_{\geq 2} ) = \\
 & (1, 2, 2, 3, 3, 4, 3, 4, 4, 5, 4, 5, 5, 5, 4, 5, 5, 6, 5, 6, 6, 6, \ldots).
\end{align*}

\begin{example}
 The point of disagreement corresponding to the value $\| 23 \|_{2} = 7$ indicated above may be illustrated via the decomposition 
\begin{equation}\label{decomposition23}
 ((2 \times 2 + 1) \times 2 + 1) \times 2 + 1 = 23 
\end{equation}
 or the decomposition 
\begin{equation}\label{another23decomposition}
 (2 \times 2 + 1) \times 2 \times 2 + 1 + 2 = 23, 
\end{equation}
 or equivalent decompositions such as $ ((2 + 2 + 1) \times 2 + 1) \times 2 + 1 = 23 $ and $ (2 + 2 + 1) \times 2 \times 2 + 1 + 2 = 23$, and a brute 
 force search may be applied to verify that $23$ cannot be expressed with fewer than $7$ possibly repeated copies of elements from $\{ 1, 2 \}$, subject 
 to the specified constraints in the definition of $\|\cdot\|_{2}$. In contrast to Equations 
 \eqref{decomposition23} and \eqref{another23decomposition}, a shortest 
 addition chain for $23$ is $ (1, 2, 3, 5, 10, 20, 23)$, with a length of $6$ according to Definition \ref{definitionlength}, so 
 that $\ell(23) = 6$. 
\end{example}

 Successive arguments $n$ such that $\| n \|_{2}$ is not equal to the length of a shortest addition chain for $n$ include 
\begin{equation}\label{notequallength}
 23, 43, 46, 59, 77, 83, \ldots 
\end{equation}
 and the above subsequence is not currently in the OEIS,  which suggests that our complexity function $\|\cdot\|_{2}$ is new and has not been  
 considered in relation to addition chains previously.  Is it true that $\ell(n) \leq \| n \|_{2}$ for all $n \geq 2$?  If so, what can be said in regard to the 
 possible sizes of differences of the form $\ell(n) - \| n \|_{2}$? Further open problems are given in the Conclusion section of this paper. 

 One of the most basic formulas for the integer complexity function is such that 
\begin{equation}\label{recursionoriginal}
 \| n \| = \min_{\substack{ d \mid n \\ m < n }} 
 \left\{ \| d \| + \left\| \frac{n}{d} \right\|, \| m \| + \| n - m \| \right\}, 
\end{equation}
 and this provides us with a recursion if we impose the conditions whereby the divisors $d$ involved in 
 Equation \eqref{recursionoriginal} are such that $d \neq 1$ 
 and $d \neq n$, letting it be understood that $m > 0$. As suggested by Steinerberger \cite{Steinerberger2014}, this recursion may be seen as 
 providing a notable instance whereby an easily stated number theory problem involving combinations of addition and multiplication to express natural 
 numbers can lead to computational difficulties. We may generalize 
 Equation \eqref{recursionoriginal} using the family of functions of the form $\|\cdot\|_{m}$ 
 considered in this paper, as below. 

\begin{proposition}
 For fixed $m \in \mathbb{N}$, the recursion 
\begin{equation}\label{recusivegeneralization}
 \| n \|_{m} = \min_{\substack{ d \mid n \\ \eta < n }} 
 \left\{ \| d \|_{m} + \left\| \frac{n}{d} \right\|_{m}, \| \eta \|_{m} + \| n - \eta \|_{m} \right\} 
\end{equation}
  holds for all sufficiently large $n$. In particular, for the binary case  with $m = 2$, the recursive relation in Equation \eqref{recusivegeneralization} 
  holds     for $n \geq 3$.  
\end{proposition}

  For a given $m \in \mathbb{N}$, the recursion in Equation \eqref{recusivegeneralization} may be proved (for sufficiently large $n$)  using the same line of  
  reasoning as in   known proofs of Equation \eqref{recursionoriginal} as in \cite{AriasdeReyna2024}.   Computations obtained from the $m = 2$ case of  
  Equation   \eqref{recusivegeneralization} are shown below. Observe that the analogues of \eqref{originalinequalities} whereby  
\begin{equation}\label{basicinequalitiesbinary}
 \| n + m \|_{2} \leq \| n \|_{2} + \| m \|_{2} 
 \ \ \ \text{and} \ \ \ \| n \times m \|_{2} \leq \| n \|_{2} + \| m \|_{2} 
\end{equation}
 follow in a direct way from the $m = 2$ case of Equation \eqref{recusivegeneralization}. 

 The $m = 2$ case of Equation \eqref{recusivegeneralization} provides a practical way of computing higher-order values of $\|\cdot\|_{2}$. To avoid possible 
 disagreements concerning offsets of initial values of OEIS sequences, by computing $\| n \|_{2}$ starting with, say, $n \geq 5$, we obtain 
\begin{align*}
 & ( \| n \|_{2} : n \in \mathbb{N}_{\geq 5} ) = \\
 & (3, 3, 4, 3, 4, 4, 5, 4, 5, 5, 5, 4, 5, 5, 6, 5, 6, 6, 7, 5, 6, 6, 6, \\
 & 6, 7, 6, 7, 5, 6, 6, 7, 6, 7, 7, 7, 6, 7, 7, 8, 7, 7, 8, 8, \ldots)
\end{align*}
 and, remarkably, no sequences currently included in the OEIS agree with the above data. This again suggests that our binary complexity function 
 $\|\cdot\|_{2}$ has not previously been considered, motivating our explorations in Section \ref{subsectionExplicit} below. 

 The last entry displayed in the above output associated with $\|\cdot\|_{2}$ 
 corresponds to the argument $47$. Remarkably, if we were to instead input
 $( \| 2 \|_{2}, \| 3 \|_{2}, \ldots, \| 46 \|_{2} )$ into the OEIS, 
 we find an agreement with the OEIS sequences {\tt A117497} and {\tt A117498}, 
 with a disagreement with $\| n \|_{2}$ occurring for the $n = 47$ argument. 

\begin{example}
 We may verify that $\| 47 \| = 8$, by verifying that $ (2 \times 2 + 1) \times (2 + 1) \times (2+1) + 2 = 47 $ and $ (2 + 2 + 1) \times (2 + 1) \times (2 
 + 1) + 2 = 47 $ provide minimal combinations of ones and twos to express $47$, again according to the restrictions associated with the complexity 
 function $\|\cdot\|_{2}$. As indicated above, this provides the first point of disagreement with {\tt A117497} and the first point of disagreement 
 with {\tt A117498}. 
\end{example}

 The sequence {\tt A117497} is defined by analogy with lengths of shortest addition chains, and {\tt A117498} provides the order of an optimized 
 combination of binary and factor methods for producing an addition chain, referring to the OEIS for details. This again reflects the close connection 
 between $\|\cdot\|_{2}$ and shortest addition chains, and illustrates the research interest in the new complexity function $\|\cdot\|_{2}$ that is the 
 subject of this paper. 

\subsection{Explicit Bounds for the Binary Complexity Function}\label{subsectionExplicit}
 While our below results may be generalized to our new family of complexity functions of the form $\|\cdot\|_{m}$ for $m \in \mathbb{N}_{\geq 2}$, we 
 focus on the $m = 2$ case for reasons suggested above. To introduce and prove lower and upper bounds for $\|\cdot\|_{2}$, we modify an approach 
 described by Arias de Reyna \cite{AriasdeReyna2024} (cf.\ \cite{AriasdeReyna2000}). Observe that a lower bound for $\| \cdot \|$ would not 
 necessarily provide a lower bound for $\| \cdot \|_{2}$, and an optimal upper bound for $\| \cdot \|_{2}$ would not provide an upper bound for $\| 
 \cdot \|$, as illustrated via the respective behaviors of $\| \cdot \|$ and $\| \cdot \|_{2}$ shown in Figure \ref{Figure1}. 

\begin{figure}
\begin{center}
\includegraphics[scale=0.25]{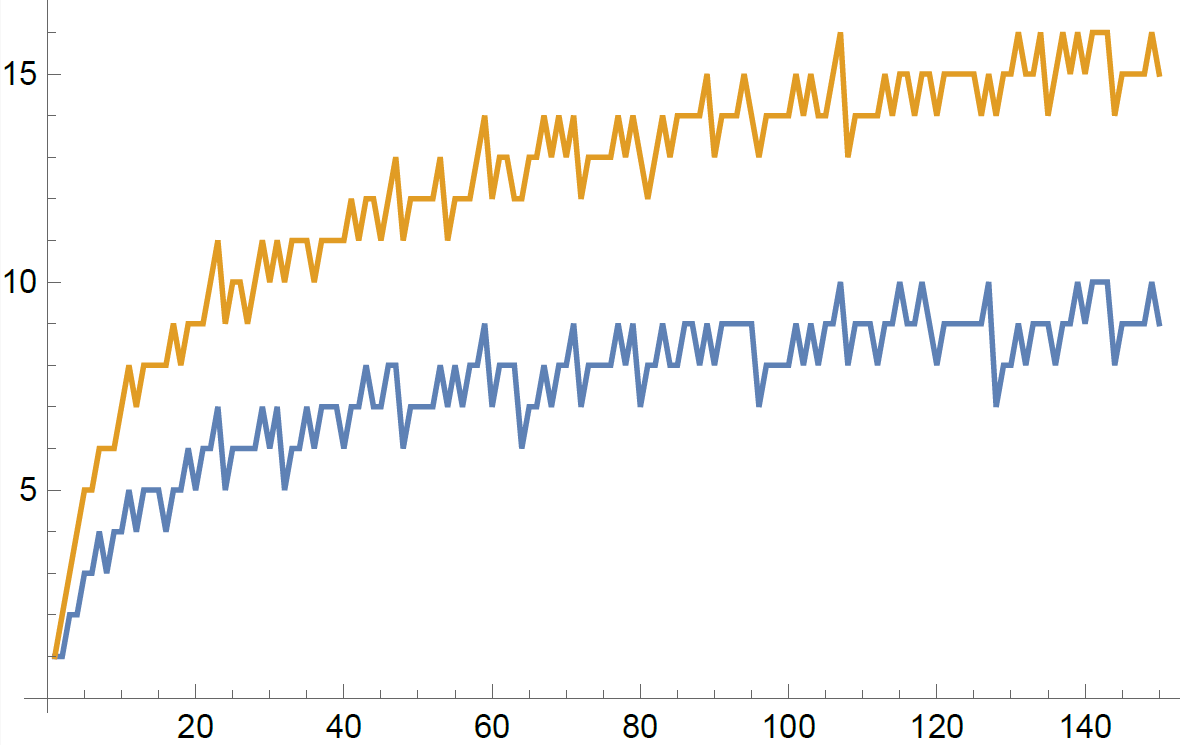}
\caption{\label{Figure1} 
 An illustration of the behavior of $\| \cdot \|$ and of $\| \cdot \|_{2}$, 
 where the lower graph corresponds to $\| \cdot \|_{2}$. } 
\end{center}
\end{figure}

 We let $\log_b(x)$ denote the base-$b$ logarithm of $x$. Arias de Reyna \cite{AriasdeReyna2024} proved that $\log_{2}(n+1) \leq \| n \|$ for all 
 natural numbers $n$. We find that it is not, in general, the case that $\log_{2}(n+1) \leq \| n \|_{2}$, with counterexamples appearing for integers of 
 the form $n = 2^{m} - 1$. 

\begin{theorem}\label{theoremlower}
 The lower bound such that $\log_{2}(n) \leq \| n \|_{2} $ holds for $n \in \mathbb{N}$. 
\end{theorem}

\begin{proof}
 The base case for $n = 1$ is equivalent to $0 \leq 1$. Inductively, we assume that the desired inequality holds for natural numbers $m \leq n$. We find that 
 the $n = 2$ case is equivalent to $1 \leq 1$, and that the $n = 3$ case is equivalent to $1.58496\ldots \leq 2$, so that the desired implications hold for $n 
 \leq 2$. For $n > 2$, the recursive version of Equation \eqref{recusivegeneralization} gives us that the expression $\| n + 1 \|_{2}$ is either equal to $\| m 
 \|_{2} + \| n + 1 - m \|_{2}$ for some integer $1 \leq m \leq \frac{n +1}{2} $ or is equal to to $\| d \|_{2} + \left\| \frac{n+1}{d} \right\|_{2}$ for 
 some divisor $d$ such that $1 < d \leq \sqrt{n + 1}$. First suppose the first case holds, with $m = 1$. From our inductive hypothesis, we thus have that 
 $1 + \log_{2}(n) \leq \| n + 1 \|_{2}$, so that $\log_{2}(n+1) \leq \| n + 1 \|_{2}$, as desired. Now, suppose that $m > 1$. By the inductive hypothesis, 
 we have that $\log_{2}(m) + \log_{2}(n + 1 - m) \leq \| m \|_{2} + \| n + 1 - m \|_{2}$. Since $ n + 1 -m \geq \frac{n +1}{2}$, we have that if $ n + 1 -m 
 = 1$, then $n = 1$, but, since $m > 1$, we have that $ n + 1 -m < n = 1$, yielding a contradiction. So, we have that $m > 1$ and $ n + 1 - m > 1$, which 
 gives us that $\log_{2}(m + (n + 1 - m)) \leq \log_{2}(m) + \log_{2}(n+1-m) \leq \| m \|_{2} + \| n + 1 - m \|_{2}$, so that $\log_{2}(n + 1) \leq \| n + 
 1 \|_{2}$, as desired. A similar argument may be applied with respect to the remaining case whereby $\|n+1\|_{2}$ is 
 of the form $\| d \|_{2} + \left\| \frac{n + 1}{d} \right\|_{2}$, since 
 $$\| n + 1 \|_{2} = \| d \|_{2} + \left\| \frac{n+1}{d} \right\|_{2} 
 \geq \log_{2} d + \log_{2}\left( \frac{n+1}{d} \right) = \log_{2}(n+1). $$
 \end{proof}

 The length $\ell(n)$ of a shortest addition chain ending in $n$ satisfies 
\begin{equation}\label{boundsl}
 \log_{2}(n) \leq \ell(n) \leq \lfloor \log_{2}(n) \rfloor + \nu_{2}(n) - 1, 
\end{equation}
 where $\nu_{2}(n)$ denotes the number of ones in the base-2 digit expansion of $n$; see \cite[\S1]{AltmanPhD} and references therein. In view of the 
 close relationship between $\ell(n)$ and $\| n \|_{2}$, this motivates our devising an analogue of the upper bound in \eqref{boundsl} for $\| \cdot \|_{2}$. 

\begin{theorem}\label{theoremupperbound}
 The upper bound such that $\| n \|_{2} \leq \lfloor \log_{2}(n) \rfloor
 + \nu_{2}(n) - 1$ holds for $n \geq 2$. 
\end{theorem}

\begin{proof}
 The $n = 2$ case is equivalent to $1 \leq 1$. Assume, inductively, that $\| m \|_{2} \leq \log_{2}(m) + \nu_{2}(m) - 1$ holds for for natural numbers 
 $2 \leq m \leq n$. The $n = 3$ case is equivalent to $2 \leq 2$, so that the desired implication holds for $n = 2$. Now, set $n > 2$. If $n+1$ is even, we 
 obtain that 
\begin{align}
\begin{split}
 \| n + 1 \|_{2} 
 & \leq \| 2 \|_{2} + \left\| \frac{n+1}{2} \right\|_{2} \\
 & = 1 + \left\| \frac{n+1}{2} \right\|_{2}. 
\end{split}\label{splitupper}
\end{align}
 Since $n > 2$, we have that $2 \leq \frac{n+1}{2} \leq n$, so that, from the inductive hypothesis and from 
 \eqref{splitupper}, we find that $\| n + 1 \|_{2} \leq 1 + \log_{2}\left( \frac{n+1}{2} \right) + \nu_{2}\left( \frac{n+1}{2} \right) -1$, so that $\| n 
 + 1 \|_{2} \leq \log_{2}\left( n+1 \right) + \nu_{2}\left( n+1 \right) - 1$. 
 Now, suppose that $n + 1$ is odd. Since $n > 3$, 
 we have that 
 $n \geq 4$, with $2 \leq \frac{n}{2} \leq n$. We then find that $\| n + 1 \|_{2} \leq 2 + \left\| \frac{n}{2} \right\|_{2}$, so that 
 $\| n + 1 \|_{2} \leq 2 + \log_{2}\left( \frac{n}{2} \right) + \nu_{2}\left( \frac{n}{2} \right) - 1$. This, in turn, gives us that $\| n + 1 \|_{2} \leq 
 \log_{2}\left( n \right) + \nu_{2}\left( n \right)$, and it follows in a direct way that $\| n + 1 \|_{2} \leq \log_{2}\left( n + 1 \right) + \nu_{2}\left( n 
 +1 \right) -1$. We have showed that $\| n \|_{2} \leq \log_{2}(n) + \nu_{2}(n) - 1$ holds for $n \geq 2$, and, by taking the integer parts of both 
 sides of the given inequality, we obtain the desired result. 
\end{proof}

\section{Toward the Determination of an Asymptotic Formula}
 In view of the close relationship between binary decompositions and addition chains, it is surprisingly difficult to prove an analogue for the binary 
 complexity function of Brauer's 1939 formula $\ell(n) \sim \log_{2}(n)$ \cite{Brauer1939}. We describe two approaches toward this problem below. 

\subsection{Toward an Analogue of Brauer's Proof}
 To mimic Brauer's proof of 
 $\ell(n) \sim \log_{2}(n)$ \cite{Brauer1939}, 
 we would want to prove the following Proposition, 
 as the desired equivalence 
\begin{equation}\label{desiredsimn2}
 \| n \|_{2} \sim \log_{2}(n)
\end{equation}
 would follow in a direct way if Conjecture \ref{conjectureBraueranalogue} below holds true. 

\begin{conjecture}\label{conjectureBraueranalogue}
 Let $r$ be a positive integer and let $s$ be a nonnegative integer, and let $r$ and $s$ be such that it is not the case that both $r = 1$ and $s = 0$. 
 Then $$ \| n \|_{2} \leq (r+1) s + 2^{r} - 2 $$ for $2^{rs} \leq n < 2^{r(s+1)} $ (cf.\ \cite{Brauer1939}). 
\end{conjecture}

 Informally, the main problem when it comes to proving the purported property highlighted as Conjecture \ref{conjectureBraueranalogue} may be explained 
 as follows. Given an addition chain, it is always possible to increase the length by $1$ by forming a new addition chain by adding \emph{any} previous 
 entry to the last entry, so as to form a new entry. However, it is not, in general, possible to follow this approach with the use of binary decompositions: 
 Informally, a combination $c$ of ones and twos may be ``buried'' within nested brackets, 
 so it would not necessarily be possible to simply multiply $c$ by $2$. 
 This is clarified below. 

 For the $r = 0$ base case, the desired inequality $\| n \|_{2} \leq 2 s$ holding for $2^{s} \leq n < 2^{s+1}$ follows from Theorem 
 \ref{theoremupperbound}. Mimicking Brauer's approach \cite{Brauer1939}, we claim that it is possible to form a binary decomposition of $n$ of length 
 at most $(r+1) s + 2^r - 2$ such at an innermost pair of brackets contains a summation with 
\begin{equation}\label{underbracecontains}
 \underbrace{1, 1+1, \ldots, 1+1+1+\cdots+1}_{\text{$2^r - 1$ terms}}
\end{equation}
 as its consecutive partial sums. Disregarding the trivial $(r, s) = (1, 0)$ case, we find that the $s = 0$ case holds, by finding that the last term in 
 \eqref{underbracecontains} evaluates as $2^{r} - 1$, with the inequalities $2^{rs} \leq n < 2^{r(s+1)}$ reducing to $1 \leq n < 2^{r}$, so that it is 
 possible to form a binary decomposition by taking the partial sum in \eqref{underbracecontains} equal to $n \in \{ 1, 2, \ldots, 2^{r} - 1 \}$. As in 
 Brauer's proof, we proceed by contradiction, and assume that the above claim does not hold, and we set $s$ as the smallest integer such that the claim is 
 not satisfied, again for $n$ such that $2^{rs} \leq n < 2^{r(s+1)}$. Dividing $n$ by $2^{r}$, writing 
\begin{equation}\label{dividingn2r}
 n = a \cdot 2^{r} + b 
\end{equation}
 for 
\begin{equation}\label{0leqbpower}
 0 \leq b < 2^{r}
\end{equation}
 and $2^{r(s-1)} \leq a < 2^{rs}$, 
 the claim holds for $s - 1 < s$ from our assumption. 
 So, there is a binary decomposition, which we denote as $B$, of $a$ of length at most 
\begin{equation}\label{Blengthatmost}
 (r + 1) (s - 1) + 2^{r} - 2, 
\end{equation}
 with \eqref{underbracecontains} as a sequence of partial sums appearing in an innermost bracket. That is, the expression $B$ refers to the value $a$, with 
 the understanding that $B$ is expressed as a combination of ones and twos according to the operational symbols specified in Definition 
 \ref{definitionmary}, so that the total number of ones and twos in $B$ is at most \eqref{Blengthatmost}, and with \eqref{underbracecontains} 
 appearing as specified. Suppose that $b > 0$. Then, from \eqref{0leqbpower}, the specified sequence of partial sums contains $b$. We expand $B$ 
 by expanding an innermost sum of the form $$ \underbrace{1+1+1+\cdots+1}_{\text{$2^r - 1$ terms}} $$ according to a partition into a sum of $b$ 
 terms and a sum of $2^{r} - 1 - b$ terms. 

 Let us write $B$ in the form 
\begin{equation}\label{writeBform}
 B = C \times \left( b + 2^r - 1 - b \right) + D 
\end{equation}
 for binary decompositions $C$ and $D$, with 
 $$ \ell(B) = 2^r - 1 + \ell(C) + \ell(D), $$
 letting it be understood that $b$ refers to an expressions of the form $1 + 1 + \cdots + 1$
 summing to $b$, and similarly for the term $2^{r} - 1 - b$. 
 We then rewrite the binary decomposition of $a$ in Equation \eqref{writeBform} as 
 $$ C \times b + C \times (2^{r} - 1 - b) + D. $$ 
 We then form a binary decomposition in the manner suggested as follows: 
\begin{align*}
 & \underbrace{2 \times 2 \times \cdots \times 2}_{r} \times C \times b + \\
 & \underbrace{2 \times 2 \times \cdots \times 2}_{r} \times C \times (2^{r} - 1 - b) + D. 
\end{align*}
 This gives us a binary decomposition of $2^{r} a$. We then obtain 
 a binary decomposition of Equation \eqref{dividingn2r} in the manner suggested below: 
\begin{align*}
 & \left( \underbrace{2 \times 2 \times \cdots \times 2}_{r} \times C + 1 \right) \times b + \\
 & \underbrace{2 \times 2 \times \cdots \times 2}_{r} \times C \times (2^{r} - 1 - b) + D. 
\end{align*}
 We find that this is of length 
\begin{align*}
 & r + \ell(C) + 1 + b + \\ 
 & r + \ell(C) + (2^{r} - 1 - b) + \ell(D), 
\end{align*}
 which we rewrite as 
\begin{align*}
 & 2 r + \ell(C) + 1 + \ell(B), 
\end{align*}
 which is less than or equal to 
\begin{align*}
 & 2 r + \ell(C) + 1 + (r + 1) (s - 1) + 2^{r} - 2, 
\end{align*}
 By analogy with Brauer's proof, we would want this to be less than or equal to 
 $(r + 1) s + 2^r - 2$, 
 but this desired property is equivalent to $0 \leq -\ell(C) - r$, which does not hold. 

 Numerical evidence suggests that Conjecture \ref{conjectureBraueranalogue} holds, but, in view of the above difficulties concerning devising an analogue 
    of Brauer's   
 proof, we leave it as an open problem to prove Conjecture \ref{conjectureBraueranalogue}. 

\subsection{Addition Chain Lengths and Binary Decomposition Lengths}
 Using the available data in this OEIS entry {\tt A003313} together with our recursion for $\| \cdot \|_{2}$, we have computed $\| n \|_{2} - \ell(n)$ for 
 $n \leq 10000$. For example, all of the values in $( \| n \|_{2} - \ell(n) : n \in \mathbb{N}_{\leq 10000})$ are in $\{ 0, 1, 2, 3 \}$, with $1$ entry equal 
 to $3$. This suggests that it may be possible to determine appropriate bounds for a function $f(n)$ such that $\| n \|_{2} - \ell(n) = f(n)$, in such a 
 way so that $f(n)$ is bounded above by an elementary function in such a way that the purported equivalence in \eqref{desiredsimn2} would be 
 immediate from Brauer's equivalence. 

 Informally, given a binary decomposition $B$, we could form an addition chain by recording consecutive partial sums appearing within an innermost pair 
 of brackets, and by then continuing in a recursive fashion according to the combination of ones and twos and operational symbols appearing outside of 
 this pair of brackets. Formalizing this notion could lead to a proof that $\ell(n) \leq \| n \|_{2}$. Conversely, given an addition chain $A$, what would 
 be an appropriate procedure to form a corresponding binary decomposition, in such a way to show that $\| n \|_{2} - \ell(n)$ is bounded above by an 
 elementary function? 

\section{Conclusion}
 We consider some further areas to explore related to $\| \cdot \|_{m}$. 

 To begin with, we encourage a full exploration of complexity functions of the form $\|\cdot\|_{m}$ for $m \in \mathbb{N}_{\geq 2}$, in addition to the 
 material above concerning the $m = 2$ case. 

 Recall the integer subsequence in \eqref{notequallength} providing successive indices such that $\| n \|_{2} \neq \ell(n)$. 
 How can this sequence be evaluated in an explicit way? What can be determined in regard to growth properties of this 
 sequence? How can the addition chains corresponding to the indices in \eqref{notequallength} be characterized without explicit or direct 
 reference to $\|\cdot\|_{2}$? 

 By analogy with the correspondence between $\|\cdot\|_{2}$ and shortest addition chains suggested in Example \ref{shortest20}, what would be an 
 appropriate unary or ternary analogue of the definition of an addition chain, in relation to the complexity functions $\|\cdot\|$ and $\|\cdot\|_{3}$? 

 Instead of studying properties of sequences of the form $(\| n \|_{m} : n \in \mathbb{N} )$ 
 for fixed $m \in \mathbb{N}$, 
 how could analogous properties be determined for sequences of the form 
 $(\| n \|_{m} : m \in \mathbb{N} )$ for fixed $n \in \mathbb{N}$? 

 Recalling the lower bound for $\| \cdot \|$ shown in \eqref{boundsoriginal}, the \emph{integer defect} $\delta(n)$ of a natural number was defined by 
 Altman and Zelinsky as $\| n \| - 3 \log_{3}(n)$ \cite{AltmanZelinsky2012}. Altman's explorations of integer defects as in \cite{Altman2016,Altman2019} 
 motivate analogous explorations of the function given by $\| n \|_{2} - \log_{2}(n)$. 

 Recall the close connection indicated in Section \ref{sectionAddition} between $\| n \|_{2}$ and $\ell(n)$. An important \cite[\S1]{AltmanPhD} open 
 question concerning addition chains is given by the \emph{Scholz--Brauer conjecture}, according to which we would have that 
 $\ell(2^{n} - 1) \leq n + \ell(n) - 1$. 
 Could our new complexity function $\| \cdot \|_{2}$ be used to give light to the Scholz--Brauer conjecture?

\vskip20pt\noindent {\bf Acknowledgements.}
 The author is grateful to acknowledge support from a Killam Postdoctoral Fellowship from the Killam Trusts. 
 The author is thankful to Karl Dilcher for very engaging and very useful discussions concerning this paper.

\end{document}